\documentclass[12pt]{amsart}

\usepackage{amssymb}
\usepackage{enumitem}

\usepackage{graphicx}

\makeatletter
\@namedef{subjclassname@2010}{%
  \textup{2010} Mathematics Subject Classification}
\makeatother

\usepackage[T1]{fontenc}

\newtheorem{thm}[equation]{Theorem}

\newtheorem*{conj*}{Conjecture}
\newtheorem*{thm*}{Theorem}
\newtheorem{rem}[equation]{Remark}

\newtheorem{lem}[equation]{Lemma}
\numberwithin{equation}{section}

\def\f{\frac}
\def\l{\left}
\def\r{\right}
\def\be{\begin{equation}}
\def\ee{\end{equation}}
\def\ba{\begin{align*}}
\def\eal{\end{align*}}
\def\inter{\intertext}

\begin{document}



\title[Finiteness of trivial solutions]{Finiteness of trivial solutions of factorial products yielding a factorial over number fields}

\author[Wataru Takeda]{Wataru Takeda}
\address{Department of Mathematics, Nagoya University\\Chikusa-ku, Nagoya 464-8602, Japan}
\email{d18002r@math.nagoya-u.ac.jp}

\date{}

\begin{abstract}
We consider a Bertrand type estimate for primes splitting completely. As one of its applications, we show the finiteness of trivial solutions of Diophantine equation about the factorial function over number fields except for the case the rational number field.\end{abstract}

\subjclass[2010]{Primary 11D45; Secondary 11D57, 11D72}

\keywords{Bertrand type estimate, Diophantine equation, finiteness theorem, number field}

\maketitle

\section{Introduction}
Let $K$ be a number field and let $\mathcal{O}_K$ be its ring of integers. 
We consider the factorial function generalized to number fields. This function $\Pi_K(x)$ is defined as
\[ \Pi_K(x)=\prod_{\mathfrak{Na}\le x}\mathfrak{Na}=\prod_{n\le x}n^{a_K(n)},\]
where {$\mathfrak{Na}$ is the norm of $\mathfrak a$ and} $a_K(n)$ is the number of ideals with $\mathfrak{Na}=n$. It is well known that $a_K(n)$ has the multiplicative property
\[a_K(mn)=a_K(m)a_K(n)\hspace{5mm}\text{ if }\gcd(m,n)=1.\] 
In the following, we will use abbreviation $a(n)$ for $a_K(n)$. If the product is empty then we assign it the value $1$.
{As a generalization,} we consider the equation \begin{equation} \label{deq} \Pi_K(l_1)\cdots\Pi_K(l_{m-1})=\Pi_K(l_m)\end{equation} for $2\le l_1\le\cdots\le l_{m-1}<l_m$. 

In the case $K=\mathbf Q$, this equation has infinitely many solutions but most of them satisfy $l_m-l_{m-1}=1$ and they are called trivial solutions. In this case, it is known that if the ABC conjecture holds, then there are only finitely many non-trivial solutions of (\ref{deq}) \cite{lu}. The $3$-tuple $(6,7,10)$ is one of such non-trivial solutions, but we do not know others. In \cite{hps}, they show that non-trivial $3$-tuple solutions of (\ref{deq}) other than $(l_1,l_2,l_3) = (6, 7, 10)$ satisfy $l_3< 5(l_2-l_1)$ and if $l_2-l_1\le 10^6$ then the only non-trivial solution to (\ref{deq}) is $(6, 7, 10)$. 

In {the} general case, a solution $(l_1,\ldots,l_m)$ is called trivial if there exists no ideals with $l_{m-1}<\mathfrak{Na}< l_m$. For example, when $K=\mathbf Q(\sqrt{-3})$, we describe the splitting of prime ideals in $\mathcal O_K=\mathbf Z\l[\f{1+\sqrt{-3}}2\r]$ as follows:
\[\begin{array}{l|l|l}
\text{prime $p$ in $\mathbf Z$}&\text{how to split in }\mathcal O_K&\text{$a(p^k)$ for $k\ge1$}\\
\hline
p\equiv1\mod 3&(p)\text{ splits completely in $\mathcal O_K$.}&a(p^k)=k+1\\
p\equiv2\mod 3&(p)\text{ is also a prime ideal of $\mathcal O_K$.}&a(p^{2k-1})=0, a(p^{2k})=1\\
p=3&(3)\text{ ramifies in $\mathcal O_K$.}&a(p^k)=1
\end{array}
\]
From the multiplicative property of $a(n)$ we can calculate $a(n)$ for all $n$. One can check that two $3$-tuples $(4,9,12)$ and $(12,247,252)$ are trivial solutions while the $3$-tuple $(16,111,117)$ is a non-trivial solution. {We do not know} whether or not there exist infinitely many solutions of (\ref{deq}). 

In this paper, we show the finiteness of trivial solutions.
\begin{thm}
\label{mtm}
For any number field $K\not=\mathbf Q$, there exist only finitely many trivial solutions of {the} Diophantine equation (\ref{deq}).
\end{thm} 
As we remarked above, there are infinitely many trivial solutions in the case $K=\mathbf Q$. On the other hand, this theorem asserts that there exist only finitely many trivial solutions for any number fields $K\not=\mathbf Q$. That is, there is an essential difference between the case $K=\mathbf Q$ and the general case.

\section{Auxiliary lemmas}
In this section we show and introduce some auxiliary lemmas to prove Theorem \ref{mtm}. The first lemma gives a necessary and sufficient condition for the existence of trivial solution.
\begin{lem}
\label{equi}
The following two statements are equivalent.
\begin{enumerate}
\item {The} $m$-tuple $(l_1,\ldots,l_m)$ is a trivial solution.
\item {Letting $l_m=\prod_p p^{r_p}$, we have} \[\Pi_K(l_1)\cdots\Pi_K(l_{m-2})=l_m^{a(l_m)}=\l(\prod_{p}p^{r_p}\r)^{\prod_{p}a(p^{r_p})}.\]
\end{enumerate}
\end{lem}
\begin{proof}
Let $(l_1,\ldots,l_m)$ be a trivial solution. Then the equation $\Pi_K(l_1)\cdots\Pi_K(l_{m-1})=\Pi_K(l_m)$ can be rewritten as \[\Pi_K(l_1)\cdots\Pi_K(l_{m-2})=\prod_{l_{m-1}<\mathfrak{Na}\le l_m }\mathfrak{Na}=l_m^{a(l_m)}.\]
The second equality follows since there exists no ideal {with norm} in the interval $(l_{m-1},l_m)$. 

Conversely, when $\Pi_K(l_1)\cdots\Pi_K(l_{m-2})=l_m^{a(l_m)}$,
we define $l_{m-1}=\max\{\mathfrak{Na}~|~\mathfrak a: \text{ideal}\}\cap[l_{m-2},l_m)$. Then 
\ba
\Pi_K(l_m)&=l_m^{a(l_m)}\Pi_K(l_{m-1})\\
&=\Pi_K(l_1)\cdots\Pi_K(l_{m-1}).\\
\end{align*}
Therefore, $(l_1,\ldots,l_m)$ be a solution. From the definition of $l_{m-1}$, there exists no ideal with their ideal norm being in the interval $(l_{m-1},l_m)$. Hence $(l_1,\ldots,l_m)$ is a trivial solution.
This proves this equivalence.
\end{proof}
\begin{rem}
{In the} case $K=\mathbf Q$, we know $a(n)=1$ for all $n$. This asserts that $(l_1,\ldots,l_m)$ is a trivial solution if and only if $l_m-l_{m-1}=1$. This does not contradict the definition of trivial solution.
\end{rem}

In 2017, Hulse and Murty gave {a} generalization of Bertrand's postulate, or Chebyshev's theorem, to number fields \cite{pp}. The original Bertrand postulate states that for any $x > 1$ there exists a prime number in the interval {$(x, 2x)$}. This is a result weaker than the prime number theorem but we can prove this without information about the zeros of the zeta function.

In the followings, we consider a Bertrand type estimate for primes splitting completely by following the way of \cite{pp}.
In \cite{pp}, they use the following effective version of prime ideal theorem given by Lagarias and Odlyzko.
Let $L/K$ be a Galois extension. For each conjugacy class $C$ of $G$, we define $\pi_{C}(x)$ by 
\[\pi_{C}(x)=\l|\{\mathfrak p \subset \mathcal O_K~|~ \mathfrak p \text{ is unramified in } L, [(\mathfrak p, L/K)]=C, \mathfrak{Np}\le x \}\r|,\]
where $[(\mathfrak p, L/K)]$ is the conjugacy class of Frobenius map corresponding to $\mathfrak p$. 
\begin{lem}[Theorem 1.3. of \cite{lo77}]
\label{hh}
Let $L/K$ be a Galois extension of number fields with $[L:\mathbf Q]=n$ and let $D_L$ be the absolute value of the discriminant of $L$. Then there exist effectively computable positive constants $c_1$ and $c_2$,
such that if $x>\exp\l(10n(\log D_L)^2\r)$ then
\[\l|\pi_C(x)- \f{|C|}{|G|}Li(x) + \f{|C|}{|G|}(-1)^{\varepsilon_L} Li(x^\beta)\r|\le c_1x \exp\l(-c_2\sqrt{\f{\log x}n}\r),\]
where $Li(x^\beta)$ only occurs if there exists an exceptional real zero $\beta$ of $\zeta_L(s)$ such that $1-(4 \log D_L)^{-1} <\beta< 1$. Also $\varepsilon_L = 0$ or $1$ depending on $L$.
\end{lem}
If $\mathfrak p$ splits completely in $L$ then the Frobenius map $(\mathfrak p, L/K)$ is trivial and $|[(\mathfrak p, L/K)]|=1$. From the definition of $\varepsilon_L$ in \cite{lo77}, we obtain $\varepsilon_L=0$. The next theorem gives a Bertrand type estimate for primes splitting completely.
\begin{thm}
\label{pp}
Let $K$ be a number field and $K^{gal}$ be the Galois closure of $K/\mathbf Q$ with $[K^{gal}:\mathbf Q]=k$ and let $D$ be the absolute value of the discriminant of $K^{gal}$.
For any $A > 1$ there exists an effectively computable constant $c(A) > 0$ depending only on $A$ such that for
$x > \exp(c(A)k(\log D)^2)$ there is a prime splitting completely with
$x<p\le Ax$.
\end{thm}
\begin{proof}
It is well known that a prime $p$ splits completely in $K$ if and only if it splits completely in the smallest Galois extension $K^{gal}$ of $\mathbf Q$ containing $K$. Without loss of generality, we assume that $K/\mathbf Q$ is a Galois extension. Let $\pi_{s.c.}(x)$ be the number of primes $p\le x$ splitting completely in $K$. 
From Lemma \ref{hh} and the above remark, we get\ba
&\pi_{s.c.}(Ax)-\pi_{s.c.}(x)\\
&>\f1{k}\l(Li(Ax)-Li(x)\r)-\f{1}{k}\l(Li(\l(Ax\r)^\beta)-Li(x^\beta)\r)-2Ac_1x\exp\l(-c_2\sqrt{\f{\log x}{k}}\r).\\
\end{align*}
It suffices to show that the right hand side is positive for $x>\exp(c(A)k(\log D)^2)$.

Stark \cite{st74} showed that if $K/\mathbf Q$ is a Galois extension and $\beta$ exists then 
\be\label{star}
1-\f1{4\log D}<\beta<1-\f{c_3}{D^{\f1k}},
\ee
where $c_3$ is an effectively computable positive constant. One can check that $Li\l((Ax)^\beta\r)-Li(x^\beta)$ is a monotonically increasing function in $\beta$ for fixed $x>\exp\l(10k(\log D)^2\r)$ and $A>1$, so we put $\beta_0=1-c_3D^{-\f1k}$. By integration by parts we find that it suffices to show 
\ba
&\f{Ax}{\log Ax}-\f{(Ax)^{\beta_0}}{\beta_0\log Ax}+\int_{(Ax)^{\beta_0}}^{Ax}\f{dt}{(\log t)^2}\\
&>\f x{\log x}-\f {x^{\beta_0}}{\beta_0\log x}+\int_{x^{\beta_0}}^{x}\f{dt}{(\log t)^2}+2Akc_1x\exp\l(-c_2\sqrt{\f{\log x}{k}}\r) 
\end{align*}
for $x>\exp\l(10k(\log D)^2\r)$ and $A>1$. 
The function $\int_{x^{\beta_0}}^{x}\f{dt}{(\log t)^2}$ also increases as $x$ increases for $x>\exp(10k(\log D)^2)$, so our goal is to show
\begin{align}
\label{axx}
\f{Ax\beta_0-(Ax)^{\beta_0}}{\beta_0\log Ax}&>\f {x\beta_0-x^{\beta_0}}{\beta_0\log x}+2Akc_1x\exp\l(-c_2\sqrt{\f{\log x}{k}}\r).\nonumber\\
\inter{{Dividing} by $\f {x\beta_0-x^{\beta_0}}{\beta_0\log x}$ this is equivalent to}
A\f{\beta_0-(Ax)^{\beta_0-1}}{\beta_0-x^{\beta_0-1}}\f{\log x}{\log Ax}&>1+\f{2A\beta_0 kc_1\log x}{\beta_0-x^{\beta_0-1}}\exp\l(-c_2\sqrt{\f{\log x}{k}}\r).
\end{align}
Now we {put} $x=\exp\l(c(A)k(\log D)^2\r)$. Then the right hand side is equal to
\be\label{right}\f{2\beta_0 Ak^2c_1c(A)(\log D)^2D^{-c_2\sqrt{c(A)}}}{\beta_0-x^{\beta_0-1}}.\ee
The denominator of (\ref{right}) is monotonically increasing and positive for $x>10$. Moreover the Minkowski bound $\f kD\le\l(\f4\pi\r)^k\f {(k!)^2}{k^{2k-1}}$ \cite{La90} leads {to}
\be\label{min}
\f k{D}\le\l(\f 4\pi\r)^k\f{(k!)^2}{k^{2k-1}}\le\f{8}{\pi^2}.\ee
The numerator of (\ref{right}) is equal to
\[
2\beta_0 Ac_1\f{c(A)}{D^{c_2\sqrt{c(A)}-3}}\l(\f{k}{D}\r)^2\f{(\log D)^2}{D}.\]
For $c(A)>4c_2^{-2}$ this function is monotonically decreasing and from (\ref{min}) we can choose $c(A)$ independent of $K$.
Therefore, {both} sides of (\ref{axx}) is decreasing for $x>\exp\l(c(A)k(\log D)^2\r)$. Also the left hand side converges to $A$ and the right {converges to} $1$ as $x$ tends to infinity. 
Thus there exists $c(A)$, independent of $K$, such that if $x>\exp\l(c(A)k(\log D)^2\r)$, then inequality (\ref{axx}) holds, that is, $\pi_{s.c.}(Ax)-\pi_{s.c.}(x)>0$.
This proves the theorem.
\end{proof}

\section{Proof of the main theorem}
In this section, we show Theorem \ref{mtm}. We {recall} the main theorem. 

\setcounter{section}{1}
\setcounter{equation}{1}
\begin{thm}
\label{main}
For any number field $K\not=\mathbf Q$, there exists only finitely many trivial solutions of Diophantine equation (\ref{deq}).
\end{thm} 

\begin{proof}[Proof of the main theorem]
Let $n=[K:\mathbf Q]$, $k=[K^{gal}:\mathbf Q]$ and let $D$ be the absolute value of the discriminant of $K^{gal}$.
We denote the minimum of the set $\{\mathfrak{Na}~|~\mathfrak a:\text{ideal of $\mathcal{O}_K$}\}\cap\mathbf Z_{>1}$ by $p_1$. From Theorem \ref{pp},  there exists $c_{p_1}$ such that there is a prime splitting completely in $(x, p_1x]$ for $x\ge \exp(c_{p_1}k(\log D)^2)$. 
Let $P_{s.c.}(x)$ be the set of all primes $p$ splitting completely in $K$ with $p\le x$. Let $q$ be a prime splitting completely such that $q\ge \exp(c_{p_1}k(\log D)^2)$ and $n^{|P_{s.c.}(q)|}>n(m-2)$.
For $q\le l_{m-2}<p_1q$, if {the} $m$-tuple $(l_1,\ldots,l_m)$ is trivial then from Lemma \ref{equi} we obtain the following prime factorization of $\Pi_K(l_1)\cdots\Pi_K(l_{m-2})$: \[\Pi_K(l_1)\cdots\Pi_K(l_{m-2})=\l(q^{r_q}\prod_{p\not=q}p^{r_{p}}\r)^{\prod_{p}a(p^{r_p})},\]
where $r_p\ge0$ for all $p$ and $r_q\ge1$.

Since $r_q\ge1$ and $a(p^{r_p})\ge n$ for all $p\in P_{s.c.}(q)$, {the product $\Pi_K(l_1)\cdots\Pi_K(l_{m-2})$ of factorial functions} needs to be {divisible} by at least $q^{n^{|P_{s.c.}(q)|}}$. 
The second smallest $q$-factor appears at $p_1q$, so this product $\Pi_K(l_1)\cdots\Pi_K(l_{m-2})$ is {divisible} by $q^{n(m-2)}$ at most for $q\le l_{m-2}<p_1q$. From the assumption $n(m-2)<n^{|P_{s.c.}(q)|}$, {the} $m$-tuple $(l_1,\ldots,l_m)$ is not trivial for all $q\le l_{m-2}<p_1q$. 
On the other hand, from Theorem \ref{pp} there exists a new prime {$q_1$ splitting completely} with $q<q_1\le p_1q$. Also, $q_1$ satisfies the conditions $q_1\ge \exp(c_{p_1}k(\log D)^2)$ and $n^{|P_{s.c.}(q_1)|}>n(m-2)$. {By the same argument as above}, there exists no trivial solutions $(l_1,\ldots,l_m)$ for all $q_1\le l_{m-2}<p_1q_1$ and there exists a new prime splitting completely $q_2$ with $q_1<q_2\le p_1q_1$.

By induction, there exists no trivial solution $(l_1,\ldots,l_m)$ for $l_{m-2}\ge q$. 
This shows the {required} finiteness.
\end{proof}
\setcounter{section}{3}

\section{{An} upper bound for trivial solutions}
Our main theorem implies the finiteness of trivial solutions for any $K\not= \mathbf Q$. Therefore, it is a new problem to find all trivial solutions of equation (\ref{deq}). We know that the constant $c(A)$ in Theorem \ref{pp} is effective, so one can give an explicit upper bound for $l_{m-2}$.

Since the constant $c(A)$ depends on $c_1$ and $c_2$, we can calculate $c(A)$ explicitly by the proof of Theorem \ref{pp}. Winckler obtained $c_1\le7.84 \times 10^{14}$ and $c_2=\f1{99}$. He actually obtained the estimate in Lemma \ref{hh} in a more concrete form. For the details for his results and computations, one can see his Ph.D. thesis \cite{wi13}.

In the following, we put $c_1=7.84\times10^{14}$ and $c_2=\f1{99}$ and calculate an upper bound for trivial solutions. Now we assume that there exists an exceptional zero $\beta$ of $\zeta_{K^{gal}}$. From the proof of Theorem \ref{main}, it suffices to calculate $c(2)$. From the proof of Theorem \ref{pp}, we need to obtain a constant $c(2)$ such that for $c>c(2)$
\be\label{suff}2\f{\beta-(2x)^{\beta-1}}{\beta-x^{\beta-1}}\f{\log x}{\log 2x}>1+\f{\beta ck^2(\log D)^2}{\beta-x^{\beta-1}}D^{-\f{\sqrt c}{99}},\ee
where $x=\exp\l(ck(\log D)^2\r)$. 
As we remarked in (\ref{star}), $1-\f1{4\log D}<\beta<1-c_3D^{-\f1k}<1$. Also it holds that for $x>1$ \[\f{\beta-(2x)^{\beta-1}}{\beta-x^{\beta-1}}>1.\]
Therefore, it suffices to show 
\be\label{ineq}\f{2ck(\log D)^2}{\log 2+ck(\log D)^2}>1+\f{4c_1ck^2(1-c_3D^{-\f1k})(\log D)^2}{1-c_3D^{-\f1k}-\exp\l(-c_3ck(\log D)^2D^{-\f1k}\r)}D^{-\f{\sqrt c}{99}}.\ee
Since $k\ge2$ and $D\ge3$, we get that the left-hand side of this inequality is greater than \[\f{4c(\log 3)^2}{\log 2+2c(\log 3)^2}\] and the right-hand side of this inequality (\ref{ineq}) is less than
\be\label{koko}
1+\f{4c_1ck^2(\log D)^2}{D^{\f1k}-c_3-D^{\f1k}\exp\l(-c_3ck(\log D)^2D^{-\f1k}\r)}D^{\f1k-\f{\sqrt c}{99}}.\ee
In the following, we consider the upper bound of (\ref{koko}).
First, we consider the numerator of (\ref{koko}). Inequality (\ref{min}) leads that
\[4c_1c\f{k^2}{D^2}\f{(\log D)^2}{D}D^{3+\f1k-\f{\sqrt c}{99}}<4c_1cD^{\f72-\f{\sqrt c}{99}}\f{64}{\pi^4}.\]
Next, we estimate the denominator.
Since $1-\f1{4\log D}<1-c_3D^{-\f1k}$ and $3\le D$, we have $c_3<\f14D^{\f1k}$. If $cc_3k(\log D)^2D^{-\f1k}\ge\log 3$, the denominator of (\ref{koko})
\begin{align*}
D^{\f1k}-c_3-D^{\f1k}\exp\l(-c_3ck(\log D)^2D^{-\f1k}\r)&\ge\f23D^{\f1k}-c_3\\
&>\f5{12}D^{\f1k}\\
&>\f5{12}.
\end{align*}
One can check that for $0<x\le\log 3$ the inequality $e^{-x}\le1-\f{2}{3\log3}x$ holds.
Therefore, if $c_3ck(\log D)^2D^{-\f1k}\le\log 3$ then
\begin{align*}
D^{\f1k}-c_3-D^{\f1k}\exp\l(-c_3ck(\log D)^2D^{-\f1k}\r)&\ge\f{2}{3\log3}c_3ck(\log D)^2-c_3\\
&\ge\l(\f{4}{3}c\log 3-1\r)c_3.\\
\inter{If we assume $c>1+c_3^{-1}$, we obtain}
D^{\f1k}-c_3-D^{\f1k}\exp\l(-c_3ck(\log D)^2D^{-\f1k}\r)&>1.
\end{align*}
Therefore, the goal is to find a constant $c(2)$ such that for $c>c(2)$  
\[\f{4c(\log 3)^2}{\log 2+2c(\log 3)^2}>1+\f{48}{5}c_1cD^{\f72-\f{\sqrt c}{99}}\f{64}{\pi^4}.\]
We find that it suffices to choose $c(2)=\max\{2.65\times10^7,1+c_3^{-1}\}$.
In the case that there does not exist an exceptional zero $\beta$ of $\zeta_{K^{gal}}$, we consider the inequality \[\f{2\log x}{\log 2x}>1+4c_1 ck^2(\log D)^2D^{-\f{\sqrt c}{99}}.\]
One can check that inequality (\ref{suff}) leads to this inequality. Thus we obtain $c(2)=\max\{2.65\times10^7,1+c_3^{-1}\}$.

From the proof of Theorem \ref{main}, we conclude that when $K \not=\mathbf Q$ is a number field, there exists no trivial solutions of $\Pi_K(l_1)\cdots\Pi_K(l_{m-1})=\Pi_K(l_m)$ for $l_{m-2}> \exp(c(2)k(\log D)^2)$, where $c(2)=\max\{2.65\times10^7,1+c_3^{-1}\}$. It is considered that $c_3=\f\pi6$ will suffice \cite{st74} thus we suggest that we can choose $c(2)=2.65\times10^7$.

\section*{Acknowledgement}
The author deeply thanks Prof. Kohji Matsumoto and Prof. Tapas Chatterjee for their precious advices and fruitful discussions.


\end{document}